\documentclass{article}
\usepackage{geometry}
\usepackage{titling}
\AtEndDocument{\bigskip{\footnotesize%
\textsc{$^*$School of Computing, DIT University, Dehradun, Uttarakhand–248009, India} \par  
  \textit{E-mail address:} \texttt{suryagiri456@gmail.com} \par

  \addvspace{\medskipamount}
}}
\usepackage{doc}

\usepackage{url}

\usepackage{graphicx}
\usepackage{epstopdf}
\usepackage{amsthm}
\usepackage{amssymb}
\usepackage{amsmath}
\usepackage{enumitem}

\usepackage{array}
\usepackage{enumitem}
\usepackage[utf8]{inputenc}
\usepackage[english]{babel}
\newtheorem{theorem}{Theorem}

\newtheorem{lemma}[theorem]{Lemma}
\newtheorem*{definition}{Definition}
\newtheorem*{remark}{Remark}

\DeclareMathOperator{\RE}{Re}

\usepackage{thmtools}
\declaretheorem[numbered=no,
name=Theorem A]{theorem A}

\usepackage{varwidth}
\usepackage{lineno, hyperref}
\usepackage{graphicx}
\usepackage{epstopdf}
\usepackage{amsmath}
\usepackage{amsthm}
\usepackage{enumitem}
\usepackage{amsthm}
\usepackage{float}
\usepackage{subcaption}
\usepackage{caption}
\usepackage{authblk}
\usepackage{abstract}

\begin{document}
\title{Zalcman Conjecture for Starlike Mappings in Higher Dimensions}
\author{Surya Giri$^{*}$  }


\date{}


	

\maketitle	

\begin{abstract}
    \noindent Counterexamples show that many results in the geometric function theory of one complex variable are not applicable for several complex variables. In this paper, we obtain sharp bounds for the Zalcman functional for $n=3$ associated with the starlike mappings defined on the unit ball in a complex Banach space and on the unit polydisk in $\mathbb{C}^n$. These results confirm the validity of the Zalcman conjecture in higher dimensions for $n=3$.
\end{abstract}
\vspace{0.5cm}

\section{Introduction}\label{sec1}
    Let $\mathcal{S}$ denote the class of analytic and univalent functions $f$ in the unit disk $\mathbb{U}$ with the series expansion
\begin{equation}\label{IniF}
    f(z)=z +\sum_{n=2}^\infty a_n z^n, \quad z\in \mathbb{U}.
\end{equation}
  In 1960, L. Zalcman proposed the conjecture that for every $f\in \mathcal{S}$,
\begin{equation}\label{FSF}
     \vert a_n^2 -a_{2n-1} \vert \leq (n-1)^2, \quad n\geq 2.
\end{equation}
   Brown and Tsao~\cite{BroTsa} first observed that this conjecture implies the famous Bieberbach conjecture, which asserts that $\vert a_n \vert \leq n$ for all $f\in \mathcal{S}$.
   It is well known that the conjecture is true for $n=2$, where it reduces to the Fekete-Szeg\"{o} functional.  Krushkal~\cite{Kru1,Kru2} established its validity for $n=3,4,5$ and 6. However, the problem remains open for $n\geq 7$. The conjecture  has also been studied for various subclasses of $\mathcal{S}$. Brown and Tsao~\cite{BroTsa} settled it for the class of starlike and typically real functions, while Ma~\cite{Ma} proved that it also holds for close-to-convex functions. A function $f\in \mathcal{S}$ is said to be starlike if and only if
   $ \RE ((z f'(z))/f(z)) >0$ for all $z\in \mathbb{U}.  $
   The class of all such functions is represented by $\mathcal{S}^*$. The following bound follows directly from~\cite[Theorem 2]{BroTsa}:
\begin{theorem A}\label{thmA}\cite{BroTsa}
    If $f\in \mathcal{S}^*$ is of the form (\ref{IniF}), then the following sharp estimate holds:
    $$\vert a_3^2 -a_5 \vert \leq 4. $$
\end{theorem A}
   Failure of the Bieberbach conjecture and the Riemann mapping theorem  in several complex variables was discovered by  Cartan~\cite{Car} and Poincar\'{e}~\cite{Poi}, respectively. Many results in univalent function theory in one complex variable cannot be extended to higher dimensions, at least without restrictions. Recently, sharp estimates for the Toeplitz determinants were obtained for various subclasses of biholomorphic mappings in higher dimensions (see~\cite{Gir,GirKum,GirKum2,GirKum3} and the references therein), which also give an extension of the bounds from one complex variable to several complex variables.

   In this study, we focus on the Zalcman conjecture and examine its validity in higher dimensions. We obtain sharp estimate of the Zalcman functional for $n=3$ for subclass of starlike mappings defined on the unit ball in a complex Banach space and on the unit polydisk in $\mathbb{C}^n$. These results extend Theorem A to higher dimensions and provides a partial confirmation of the Zalcman conjecture for $n=3.$

   Let $X$ be a complex Banach space with norm $\|\cdot\|$ and $\mathbb{C}^n$ denote the space of $n$ complex variables, written as $z=(z_1,z_2,\cdots , z_n)'$. The unit ball in $X$ is given by $\mathbb{B}= \{ z\in X : \|z\|<1\}$ and $\mathbb{U}^n$ denotes the unit polydisk in $\mathbb{C}^n$.  The boundary and distinguished boundary of $\mathbb{U}^n$  are denoted by $\partial \mathbb{U}^n$ and $\partial_0 \mathbb{U}^n$, respectively. For any $z \in X\setminus\{0\}$, define
   $$ T_{z}=\{ l_z \in L(X,\mathbb{C}): l_z(z) = \|z\|, \|l_z \|=1\}, $$
   where $L(X,Y)$ denotes the space of continuous linear operator from $X$ to a complex Banach space $Y$. By Hahn-Banach theorem, the set $T_z$ is non-empty. Let $\mathcal{H}(\mathbb{B})$ denote the set of holomorphic mappings from $\mathbb{B}$ into $X$. It is well known that if $f \in \mathcal{H}(\mathbb{B})$, then
   $$ f(w) = \sum_{n=0}^\infty \frac{1}{n!} D^n f(z) ((w -z)^n ) $$
   for all $w$ in some neighbourhood of $z \in \mathbb{B}$, where $D^n f(z)$ is the $n$th-Fr\'{e}chet derivative of $f$ at $z$, and for $n \geq 1$,
   $$ D^n f(z)((w-z)^k)= D^n f(z)  \underbrace{( w-z, w-z, \cdots, w-z) }_\text{ k -times}.$$
   Here, $D^n f(z)$ is a bounded symmetric $n$-linear mapping from $ \prod_{j=1}^n X$ into $X$.  On a bounded circular domain $\Omega \subset \mathbb{C}^n$, the first and the $m^{th}$ Fr\'{e}chet derivative of a holomorphic mapping $f : \Omega \rightarrow X$  are written by
    $ D f(z)$ and $D^m f(z) (a^{m-1},\cdot)$, respectively. The matrix representations are
\begin{align*}
    D f(z) &= \bigg(\frac{\partial f_j}{\partial z_k} \bigg)_{1 \leq j, k \leq n}, \\
    D^m f(z)(a^{m-1}, \cdot) &= \bigg( \sum_{k_1,k_2, \cdots, k_{m-1}=1}^n  \frac{ \partial^m f_j (z)}{\partial z_k \partial z_{k_1} \cdots \partial z_{k_{m-1}}} a_{k_1} \cdots a_{k_{m-1}}   \bigg)_{1 \leq j,k \leq n},
\end{align*}
   where $f(z) = (f_1(z), f_2(z), \cdots f_n(z))', a= (a_1, a_2, \cdots a_n)'\in \mathbb{C}^n.$

    A mapping $f\in \mathcal{H}(\mathbb{B})$ is said to be normalized if $f(0)=0$ and $Df(0)=I$, where $I \in L(X,X)$ is the identity operator. It is said to be biholomorphic if it is invertible and its inverse $f^{-1}$ is holomorphic on $f(\mathbb{B})$. If for each $z \in \mathbb{B}$, $Df(z)$ has a bounded inverse, the mapping $f$ is said to be locally biholomorphic. By $\mathcal{S}(\mathbb{B})$, we denote the class of normalized biholomorphic mappings from the unit ball $\mathbb{B}$ into $X$.
\begin{definition}\cite{Suf,Suf2}
    Let $f: \mathbb{B} \rightarrow X$ be a normalized locally biholomorphic mapping. Then $f$ is said to be starlike on $\mathbb{B}$ if and only if
    $$ \RE (l_z [Df(z)]^{-1}f(z))> 0, \quad x\in \mathbb{B}\setminus \{0\}, \quad l_z \in T_z. $$
    In the case $X=\mathbb{C}^n$ and $\mathbb{B}=\mathbb{U}^n$, this reduces to
    $$\RE \left(\frac{f_k(z)}{z_k} \right)>0 , \quad z \in \mathbb{U}^n\setminus\{0\},$$
    where $f(z)=(f_1(z),f_2(z),\cdots, f_n(z))'=[Df(z)]^{-1}f(z)$ and $k$ satisfies $\vert z_k \vert = \|z\|=\max_{1\leq j\leq n} \{\vert z_j\vert \}$.
    The class of all starlike mappings on $\mathbb{B}$ and $\mathbb{U}^n$ is denoted by $\mathcal{S}^*(\mathbb{B})$ and $\mathcal{S}^*(\mathbb{U}^n)$, respectively.
\end{definition}

   Let $\mathcal{P}$ be the class of analytic functions $p$ in $\mathbb{U}$ satisfying $p(0)=1$ and $\RE p(z)>0$ for all $z\in \mathbb{U}$.  Bounds on the coefficients of functions in $\mathcal{P}$ play an important role in coefficient problems.
   The following lemma is used to prove our main results.
\begin{lemma}~\cite{ChoRav}\label{lm3}
   Let $p(z)=1+\sum_{n=1}^\infty p_n z^n \in \mathcal{P}$. Then
   $$ \vert p_n \vert\leq 2\;\;\; \text{and}\;\;\;  \left\vert p_n - p_m p_{n-m}\right\vert \leq  2, \quad n\geq 2,\;\; 1\leq m \leq n-1. $$
\end{lemma}

   This following results provide  sharp bound of the Zalcman functional for the class of starlike mappings in higher dimensions.
\begin{theorem}\label{thmB}
    Let $f\in \mathcal{H}(\mathbb{B},\mathbb{C})$ with $f(0)=1$ and $F(z)= z f(z) \in \mathcal{S}^*(\mathbb{B})$.
   Then
\begin{equation*}
\begin{aligned}
     \bigg\vert \bigg( \frac{ l_z (D^3 F(0) (z^3))}{3! \vert\vert z \vert\vert^3} \bigg)^2 &- \bigg(\frac{ l_z (D^5 F(0) (z^5))}{5! \vert\vert z \vert\vert^5}\bigg)   \bigg\vert  \leq 4.
\end{aligned}
\end{equation*}
   The bound is sharp.
\end{theorem}
\begin{proof}
  Let $z\in X\setminus \{ 0 \}$ be fixed and set $z_0 = \frac{z}{\|z \|}$.  Define $ h : \mathbb{U} \rightarrow \mathbb{C}$ such that
\begin{equation*}
    h(\zeta) = \left\{ \begin{array}{ll}
     \dfrac{\zeta}{ l_z ([D F(\zeta z_0)]^{-1} F( \zeta z_0) )}, & \zeta \neq 0, \\ \\
    1, & \zeta =0.
    \end{array}
    \right.
\end{equation*}
   Then $h \in \mathcal{H}(\mathbb{U})$ with $h(0)=1  $. Moreover, we obtain
\begin{align*}
   h(\zeta) = \frac{\zeta}{l_z ([D F(\zeta z_0)]^{-1} F( \zeta z_0) ) } = &\frac{\zeta}{l_{z_0} ([D F(\zeta z_0)]^{-1} F( \zeta z_0) ) } \\
             =& \frac{\| \zeta z_0 \| }{l_{ \zeta z_0} ([D F(\zeta z_0)]^{-1} F( \zeta z_0) ) }.
\end{align*}
    Given that $F\in \mathcal{S}^*(\mathbb{B})$, it follows from  the definition that
\begin{equation}\label{Carat}
     \RE h(\zeta) >0 , \quad \zeta \in \mathbb{U}.
\end{equation}
   Using the same technique as in \cite[Theorem 7.1.14]{Pfa}, we get
    $$ [D F(z)]^{-1} = \frac{1}{f(z)} \bigg( I - \frac{\frac{z D f(z)}{f(z)}}{1 + \frac{D f(z) z}{f(z)}} \bigg),  \quad z\in \mathbb{B}. $$
   Noting that $F(z)=zf(z)$, the above equation leads us to
    $$ [D F(z)]^{-1} F(z) = z \bigg( \frac{z f(z) }{f(z) + D f(z) z} \bigg),  $$
    which immediately implies
\begin{equation}\label{newe}
   \frac{\| z\|}{l_z ([D F(z)]^{-1} F(z))} = 1 + \frac{D f(z) z}{f(z)} .
\end{equation}
   In view of (\ref{newe}), we get
\begin{equation}\label{accr}
    h(\zeta) = \frac{\| \zeta z_0 \| }{l_{ \zeta z_0} ([Df(\zeta z_0)]^{-1} f( \zeta z_0) ) }  =  1 + \frac{D f(\zeta z_0)\zeta z_0}{f(\zeta z_0)}.
\end{equation}
   From the series expansions of $h(\zeta)$ and $f(\zeta z_0)$, it follows that
\begin{align*}
   \bigg(1 + & h'(0) \zeta  + \frac{h''(0)}{2} \zeta^2 + \cdots \bigg)\bigg( 1 + Df(0)(z_0) \zeta + \frac{ D^2 f(0)(z_{0}^2)}{2} \zeta^2 + \cdots \bigg)  \\
   & =\bigg( 1 + Df(0)(z_0) \zeta + \frac{ D^2 f(0)(z_{0}^2)}{2} \zeta^2 + \cdots \bigg)+ \bigg( Df(0)(z_0) \zeta +  D^2 f(0)(z_{0}^2)\zeta^2 + \cdots \bigg).
\end{align*}
   Upon comparing the homogeneous expansions, we obtain
\begin{equation*}\label{eqhf}
\begin{aligned}
    h'(0) &= D f(0)(z_0),\;\; \frac{h''(0)}{2} =  D^2 f(0)(z_0^2) - (D f(0)(z_0))^2,\\
    \frac{h'''(0)}{3!} &= \frac{1}{2}\left(2 (Df(0)(z_0))^3 - 3 (D f(0)(z_0))( D^2 f(0) (z_0^2)) + D^3 f(0) (z_0^3)\right),\\
\end{aligned}
\end{equation*}
   and
\begin{align*}
    \frac{h''''(0)}{4!}
                     &=  \frac{1}{6} \Big( D^4 f(0) (z_0^4)  -6 ( Df(0)(z_0))^4+12 ( Df(0)(z_0))^2 (D^2 f(0)(z_0^2)) -3 (D^2 f(0)(z_0^2))^2\\
                      & \;\;-4 ( Df(0)(z_0)) (D^3 f(0) (z_0^3)) \Big).
\end{align*}
   That is
\begin{equation}\label{eqhf2}
\begin{aligned}
\left.
\begin{array}{ll}
     h'(0) \|z\|&= D f(0)(z), \;\;    \dfrac{h''(0)}{2}\| z\|^2 =  D^2 f(0)(z^2) - (D f(0)(z))^2, \\  \\
    \dfrac{h'''(0)}{3!} \|z\|^3 &=  \dfrac{1}{2}\left (2 Df(0)(z))^3 - 3 (D f(0)(z)) (D^2 f(0) (z^2)) +  D^3 f(0) (z^3)\right) ,\\  \\
     \dfrac{h''''(0)}{4!} \|z\|^4&= \dfrac{1}{6} \bigg(D^4 f(0) (z^4)-6 ( Df(0)(z))^4+12 ( Df(0)(z))^2 (D^2 f(0)(z^2))  \\
                             & \;\;-3 (D^2 f(0)(z^2))^2 -4 ( Df(0)(z)) (D^3 f(0) (z^3))\bigg).
\end{array}
\right\}
\end{aligned}
\end{equation}
     Since $F(z) = z f(z)$, we have
\begin{align*}
     \frac{ D^3 F(0) (z^3)}{3! } &=  \frac{ D^2 f(0) (z^2)}{2! } z \;\;\; \text{and}\;\;\;   \frac{ D^5 F(0) (z^5)}{5! } =  \frac{ D^4 f(0) (z^4)}{4! } z,
\end{align*}
   which yields
\begin{align*}
     \frac{l_z (D^3 F(0) (z^3))}{3! } =  \frac{ D^2 f(0) (z^2)}{2! } \|z\|   \;\;\; \text{and}\;\;\;
           \frac{l_z( D^5 F(0) (z^5))}{5! } =  \frac{ D^4 f(0) (z^4)}{4! }\|z\|.
\end{align*}
    Using the relations from (\ref{eqhf2}), we get
\begin{align*}
     \frac{l_z (D^3 F(0) (z^3))}{3! \|z\|^3} &= \frac{1}{2}\bigg(\frac{h''(0)}{2}+ (h'(0))^2 \bigg),\\
        \frac{l_z( D^5 F(0) (z^5))}{5! \|z\|^5} 
        &=  \frac{1}{4!}\bigg((h'(0))^4+ 6 (h'(0))^2\left(\frac{h''(0)}{2} \right)+8 h'(0) \left(\frac{h'''(0)}{3!}\right)+3 \left(\frac{h''(0)}{2}\right)^2\\
        &\quad+ 6 \left(\frac{h''''(0)}{4!}\right) \bigg).
\end{align*}
  These expressions together yield
\begin{align*}
   &  \bigg\vert  \bigg( \frac{ l_z (D^3 F(0) (z^3))}{3! \vert\vert z \vert\vert^3} \bigg)^2 - \bigg(\frac{ l_z (D^5 F(0) (z^5))}{5! \vert\vert z \vert\vert^5}\bigg) \bigg\vert \\
     &=\frac{1}{24} \bigg\vert 5 (h'(0))^4+ 6 (h'(0))^2 \Big(\frac{h''(0)}{2} \Big)- 8 h'(0) \Big(\frac{h'''(0)}{3!}\Big)+3 \Big(\frac{h''(0)}{2}\Big)^2- 6 \Big(\frac{h''''(0)}{4!}\Big) \bigg\vert \\
     &=\frac{1}{24} \bigg\vert 5 (h'(0))^2\Big( (h'(0))^2 - \frac{h''(0)}{2}\Big)+ 11 h'(0) \Big(\frac{h'(0) h''(0)}{2} -\frac{h'''(0)}{3!} \Big)+ 3 \Big( \Big(\frac{h''(0)}{2}\Big)^2 \\
     &\quad- \frac{h''''(0)}{4!}\Big) +3 \Big( \frac{h'(0) h'''(0)}{3!} - \frac{h''''(0)}{4!} \Big)\bigg\vert.
\end{align*}
    By (\ref{Carat}), we have $h\in \mathcal{P}$. Applying Lemma \ref{lm3} together with the triangle inequality in the above equality, we obtain the required result.

     To establish the sharpness part, let us consider the function $\tilde{F}$ defined by
\begin{equation*}
    \tilde{F}(z) = \frac{z}{1-(l_{u}(z))^2} , \quad z\in \mathbb{B}, \quad \vert\vert u \vert\vert=1.
\end{equation*}
   It is clear that $ \tilde{F}(z)\in \mathcal{S}^*(\mathbb{B})$. For this $\tilde{F}$, we have
  $$  \frac{D^3  \tilde{F}(0) (z^2)}{2!}=  3 (l_u(z))^2 z \;\;\; \text{and} \;\;\; \frac{D^5  \tilde{F}(0) (z^3)}{5!} = 5(l_u (z))^4 z,$$
     which in turn gives
  $$  \frac{l_z(D^3  G(0) (z^3))}{3!}= 3 (l_u(z))^2 \|z\| \;\;\; \text{and} \;\;\; \frac{l_z (D^5  G(0) ( z^5 ))}{5!} =  5 (l_u (z))^4 \| z \|, $$
   respectively. Taking $z = r u$ $(0< r <1)$, we obtain
\begin{equation*}\label{cftB}
     \frac{l_z(D^3  \tilde{F}(0) (z^3))}{3! \|z\|^3}=  3 \;\;\; \text{and} \;\;\;  \frac{l_z (D^5  \tilde{F}(0) ( z^5) ) }{5! \| z \|^5}  =  5.
\end{equation*}
   Consequently, for the mapping $\tilde{F}$, we have
\begin{align*}
   \bigg\vert \bigg( \frac{ l_z (D^3 \tilde{F}(0) (z^3))}{3! \vert\vert z \vert\vert^3} \bigg)^2 - \frac{ l_z (D^5 \tilde{F}(0) (z^5))}{5! \vert\vert z \vert\vert^5}  \bigg\vert =4,
\end{align*}
  which completes the sharpness part of the bound.
\end{proof}

\begin{theorem}\label{ThmUn1}
    Let $f \in \mathcal{H}(\mathbb{U}^n, \mathbb{C})$ with $f(0)=1$  and $F(z) =  z f(z) \in \mathcal{S}^*(\mathbb{U}^n)$. Then
\begin{equation}\label{mnresult}
\begin{aligned}
  \bigg\| \dfrac{1}{3 !} D^3 F(0) \bigg( z^2, \dfrac{D^3 F(0)(z^3)}{3!} \bigg) - \dfrac{D^5 F(0) (z^5)}{5! }  \bigg\| \leq 4 \|z\|^5 , \quad z \in \mathbb{U}^n .
\end{aligned}
\end{equation}
   The inequality is sharp.
\end{theorem}
\begin{proof}
   For $z \in \mathbb{U}^n \setminus \{ 0 \}$, let $z_0 = \frac{z}{\| z\|}$ . Define  $h_k : \mathbb{U} \rightarrow \mathbb{C}$ given by
\begin{equation*}\label{hkzeta}
   h_k (\zeta) =
\left\{
\begin{array}{ll}
     \dfrac{\zeta z_k}{q_k (\zeta z_0) \| z_0 \|}, & \zeta \neq 0,\\
     1 , & \zeta =0,
\end{array}
\right.
\end{equation*}
   where $q(z) = [D F(z)]^{-1} F(z)$ and $k$ satisfies $\vert z_k \vert = \| z\| = \max_{1 \leq j \leq n} \{ \vert z_j \vert \}$.
   Since $\RE ([D F(z)]^{-1} F(z)) >0$, it follows that
\begin{equation}\label{Carat2}
   \RE h_k (\zeta) > 0.
\end{equation}
   Moreover, using (\ref{accr}), we obtain
   $$ h_k (\zeta) =  1 + \frac{D f( \zeta z_0) \zeta z_0}{f(\zeta z_0)}. $$
   Comparison of the corresponding homogeneous terms in the Taylor expansions of $h_k$ and $f$ yields
\begin{equation}\label{use0}
\begin{aligned}
\left.
\begin{array}{lll}
    h'_k (0) =& D f(0)(z_0), \quad \dfrac{h''_k (0)}{2} = D^2 f(0) (z_0^2 ) - (D f(0) (z_0))^2, \\ \\
    \dfrac{h'''_k(0)}{3!} =& ( Df(0)(z_0))^3 - \dfrac{3}{2}D f(0)(z_0) D^2 f(0) (z_0^2)  + \dfrac{D^3 f(0) (z_0^3)}{2}, \\ \\
   \dfrac{h''''_k(0)}{4!}=&\dfrac{1}{6} \bigg( D^4 f(0) (z_0^4) -6 ( D f(0)(z_0))^4+12 ( Df(0)(z_0))^2 (D^2 f(0)(z_0^2)) \\ \\
                          & -3 (D^2 f(0)(z_0^2))^2   -4 ( D f(0)(z_0)) (D^3 f(0) (z_0^3))\bigg).
\end{array}
\right\}
\end{aligned}
\end{equation}
   Moreover, from the identity $F(z_0) =  z_0 f(z_0)$, we have
\begin{equation}\label{use}
    \frac{D^3 F_k(0) (z_0^3)}{3!} = \frac{D^2 f(0) (z_0^2)}{2!} \frac{z_k}{ \| z\|}, \quad  \frac{D^4 F_k(0) (z_0^4)}{4!} =\frac{ D^3 f(0) (z_0^3)}{3! }  \frac{z_k}{ \| z\|}
\end{equation}
   and
\begin{equation}\label{use2}
     \frac{D^5 F_k(0) (z_0^5)}{5!}   =\frac{ D^4 f(0) (z_0^4)}{4! }     \frac{z_k}{ \| z\|}.
\end{equation}
     In view of (\ref{use}) and (\ref{use2}), we obtain
\begin{align*}
    \bigg\vert  \frac{1}{3 !} D^3 F_k(0) \bigg( z_0^2, \frac{D^3 F(0)(z_0^3)}{3!}& \bigg) \frac{\| z\| }{z_k} - \frac{D^5 F_k(0) (z_0^5)}{5!} \frac{\| z\| }{z_k}\bigg\vert
    \\&=\Big\vert  \frac{1}{3 !}D^3 F_k(0) \bigg( z_0^2,   \frac{ D^2 f(0) (z_0^2)}{2! } z_0 \bigg) \frac{\| z\| }{z_k} -  \frac{ D^4 f(0) (z_0^4)}{4! }  \bigg\vert \\
       &= \Big\vert \frac{ D^2 f(0) (z_0^2)}{2! } \frac{1}{3 !} D^3 F_k(0) ( z_0^2,  z_0 ) \frac{\| z\| }{z_k} -  \frac{ D^4 f(0) (z_0^4)}{4! }  \bigg\vert  \\
       &= \Big\vert  \frac{ D^2 f(0) (z_0^2)}{2! } \frac{1}{3 !} D^3 F_k(0) ( z_0^3 ) \frac{\| z\| }{z_k} -  \frac{ D^4 f(0) (z_0^4)}{4! } \bigg\vert \\
       &= \bigg\vert \bigg( \frac{ D^2 f(0) (z_0^2)}{2! }  \bigg)^2 -  \frac{ D^4 f(0) (z_0^4)}{4! } \bigg\vert.
\end{align*}
  Applying (\ref{use0}) in the above equation, we get
\begin{align*}
    \bigg\vert& \frac{1}{3!}D^3 F_k(0) \bigg(z_0^2,  \frac{D^3 F(0)(z_0^3)}{3!} \bigg) \frac{\| z\| }{z_k} - \frac{D^5 F_k(0) (z_0^5)}{5!} \frac{\| z\| }{z_k}   \bigg\vert\\
    &=\frac{1}{24} \bigg\vert 5 (h_k'(0))^4+ 6 (h_k'(0))^2 \Big(\frac{h_k''(0)}{2} \Big)- 8 h_k'(0) \Big(\frac{h_k'''(0)}{3!}\Big)+3 \Big(\frac{h_k''(0)}{2}\Big)^2- 6 \Big(\frac{h_k''''(0)}{4!}\Big) \bigg\vert \\
     &=\frac{1}{24} \bigg\vert 5 (h_k'(0))^2\Big( (h_k'(0))^2 - \frac{h_k''(0)}{2}\Big)+ 11 h_k'(0) \Big(\frac{h_k'(0) h_k''(0)}{2} -\frac{h_k'''(0)}{3!} \Big)+ 3 \Big( \Big(\frac{h_k''(0)}{2}\Big)^2 \\
     &\quad- \frac{h_k''''(0)}{4!}\Big) +3 \Big( \frac{h_k'(0) h_k'''(0)}{3!} - \frac{h_k''''(0)}{4!} \Big)\bigg\vert.
\end{align*}
    By (\ref{Carat2}), we have $h_k\in \mathcal{P}$. Applying triangle inequality in the above equation and using Lemma~\ref{lm3}, we deduce that
$$ \bigg\vert \frac{1}{3!}D^3 F_k(0) \bigg(z_0^2,  \frac{D^3 F(0)(z_0^3)}{3!} \bigg) \frac{\| z\| }{z_k} - \frac{D^5 F_k(0) (z_0^5)}{5!} \frac{\| z\| }{z_k}   \bigg\vert
    \leq 4.$$
    If $z_0 \in \partial_0 \mathbb{U}^n$, then
    $$ \bigg\vert \frac{1}{3!}D^3 F_k(0) \bigg(z_0^2,  \frac{D^3 F(0)(z_0^3)}{3!} \bigg)  - \frac{D^5 F_k(0) (z_0^5)}{5!}   \bigg\vert
    \leq 4.$$
     Since
     $$ \frac{1}{3!}D^3 F_k(0) \bigg(z_0^2,  \frac{D^3 F(0)(z_0^3)}{3!} \bigg)  - \frac{D^5 F_k(0) (z_0^5)}{5!}, \quad k=1,2,3,\cdots, n, $$
    are holomorphic functions on $\overline{\mathbb{U}}^n$, by the maximum modulus theorem on the unit polydisk, we derive
   $$ \bigg\| \frac{1}{3!}D^3 F(0) \bigg(z^2,  \frac{D^3 F(0)(z^3)}{3!} \bigg)  - \frac{D^5 F(0) (z^5)}{5!}   \bigg\|
    \leq 4 \|z\|^5.$$

     To verify the sharpness of the bound, consider the function
\begin{equation}\label{extUnn}
    F(z)= \bigg( \frac{z_1}{1-z_1^2}, \frac{z_n}{1-z_1^2}, \cdots, \frac{z_n}{1-z_1^2},\bigg), \quad z\in \mathbb{U}^n.
\end{equation}
   It can be easily verified that $(D[F(z)]^{-1}F(z))\in \mathcal{S}^*(\mathbb{U}^n)$. Taking $z=(r,0,\cdots,n)$ $(0<r<1)$ in (\ref{extUnn}), the equality case holds in (\ref{mnresult}), which establishes the sharpness of the bound.
\end{proof}
\begin{remark}
   In the special case $X=\mathbb{C}$ and $\mathbb{B}=\mathbb{U}$, Theorem~\ref{thmB} and Theorem~\ref{ThmUn1} become equivalent to Theorem A.
\end{remark}
\section*{Declarations}
\subsection*{Conflict of interest}
	The authors declare that they have no conflict of interest.
\subsection*{Data Availability} Not Applicable.

\end{document}